\documentclass[12pt]{article}
\usepackage[T1]{fontenc}
\usepackage{amsfonts,amssymb,amsmath,textcomp,amsthm}
\usepackage{ae}    
\usepackage{hyperref}
\usepackage{geometry}
\usepackage[dvipsnames]{xcolor}

\newtheorem{thm}{Theorem}[section]
\newtheorem{prop}[thm]{Proposition}
\newtheorem{cor}[thm]{Corollary}
\newtheorem{conj}[thm]{Conjecture}
\newtheorem{lem}[thm]{Lemma}
\theoremstyle{definition}
\newtheorem{defi}[thm]{Definition}
\newtheorem{ex}[thm]{Example}

\def\blfootnote{\xdef\@thefnmark{}\@footnotetext} 

\date{}

\author{\normalsize{A. GRISHKOV, M. RASSKAZOVA, AND G. SOUZA DOS ANJOS\footnote{This study was financed in part by the Coordena\c c\~ao de Aperfei\c coamento de Pessoal de N\'ivel Superior - Brasil (CAPES) - Finance Code 001}}}

\title{\textbf{\Large{FREE BOL LOOPS OF EXPONENT TWO}}}

\begin{document}

\maketitle

\begin{abstract} 
\noindent{}A Bol  loop  is a loop that satisfies the identity $x((yz)y)=((xy)z)y$. In this paper, we give a construction of the free Bol loops of exponent two. We define a canonical form of all their elements and describe their multiplication law based on this form.

\end{abstract}


\section{Introduction}

A \emph{loop} consists of a nonempty set $L$ with a binary operation $*$ such that, for each $ a, b \in L$, the equations $ a * x  =  b $ and $ y * a =  b$ have unique solutions for $x,y \in L$, and there exists an \emph{identity element} $1\in L$ satisfying $ 1 * x = x = x *  1$, for any $ x \in L $. 
A \emph{(right) Bol  loop}  is a loop that satisfies the (right) Bol identity 

\begin{equation}
\label{bol}
x((yz)y)=((xy)z)y.
\end{equation}

One of the most interesting subvarieties of Bol loops is the variety ${\bf B}_2$ of Bol loops of exponent two. Every loop in ${\bf B}_2$ is a Bruck loop, i.e., a Bol loop with the automorphic inverse property ($(xy)^{-1} = x^{-1}y^{-1}$, for every $x,y$ in the loop). Many constructions of non-associative loops of ${\bf B}_2$ can be found in the literature (see \cite{Ki,KN} for example), the minimal such loop has order $8$. Some of the most important problems involving loops of ${\bf B}_2$ are those related to solvability and existence of simple loops (see \cite{Ach,BB,Na,Ga}). In \cite{Ga}, a class of non-associative simple Bol loops of exponent $2$ was constructed. The smallest loop in this class, which is also the smallest non-associative simple loop in ${\bf B}_2$ (\cite[Theorem 3]{BB}), has order $96$.

In this paper, we give a construction of free objects in the variety ${\bf B}_2$. Let $B(X)$ be the free Bol loop of exponent two with free set of generators $X$. We construct a subset $R(X)$ of $B(X)$ such that every element $b\in B(X)\setminus \{1\}$ has the canonical form  $b=(...(b_1b_2)b_3...)b_m)b_{m-1})...)b_2)b_1$, where $b_i\in R(X)$ and $b_i\not = b_{i+1}$, for all $i$, and then we describe the multiplication law of $B(X)$ based on this form.  Furthermore, we prove that the nuclei and the center of $B(X)$ are trivial.

\section{Preliminaries}

Let $L$ be a loop and $x\in L$. The bijections $L_x,R_x: L \to L$ defined by $(y)L_x = xy$ and $(y)R_x = yx$ are called the \emph{left and right translations} of $x$ in $L$, respectively. The \emph{right multiplication group} of $L$ is the group $Mlt_r(L) = \langle R_x\,|\,x\in L \rangle$ and the \emph{right inner mapping group} of $L$ is $Inn_r(L) =  \{ \phi \in Mlt_r(L)\,|\, (1)\phi = 1\}$. The subgroup $Inn_r(L)$ of $Mlt_r(L)$ is \emph{core-free}, i.e., the only subgroup of $Inn_r(L)$ that is normal in $Mlt_r(L)$ is the trivial subgroup $\{I_d\}$, where $I_d$ is the identity mapping of $L$.

The \textit{left}, \textit{middle} and \textit{right nuclei} of $L$, denoted respectively by $ N_ \lambda (L), N_ \mu (L) $ and $ N_ \rho (L) $, are defined by: 

\begin{center}
$N_\lambda(L) = \{a\in L\,|\,\, a(xy) = (ax)y \,\, \forall \,\, x,y \in L\}$,
\\
$N_\mu(L) = \{a\in L\,|\,\, x(ay) = (xa)y \,\, \forall \,\, x,y \in L\}$,
\\
$N_\rho(L) = \{a\in L\,|\,\, x(ya) = (xy)a \,\, \forall \,\, x,y \in L\}$.
\end{center}

The \textit{nucleus} of $L$ is defined by $N(L) = N_\lambda(L) \cap N_\mu(L) \cap N_\rho(L)$ and the \textit{center} of $L$ is the set $\mathcal{Z}(L) = \{a\in N(L)\,|\,\, ax = xa \,\, \forall \,\, x \in L\} $. The nuclei of $L$ are subgroups of $L$ and the center of $L$ is an abelian subgroup of $L$.

Bol loops are loops that satisfy the identity \eqref{bol}. This class of loops contains Moufang loops and groups. Furthermore, Bol loops are power-associative and right alternative, and have the right inverse property. Other basic facts from loop theory and Bol loops can be found in \cite{B71,P90}.

The Baer correspondence (\cite{Ba}) is an important tool in the study of Bol loops (cf. \cite{Ach}). From it, we obtain that Bol loops are related to twisted subgroups, as we can see in the next proposition. A subset $K$ of a group $G$ is called a \emph{twisted subgroup} of $G$ if $1\in K$ and $x^{-1},xyx \in K$, for all $x,y\in K$.

\begin{prop}
\label{prop1}(\cite[Proposition $5.2$]{FKP06})
Let $(G,H,B)$ be a Baer triple, i.e., $G$ is a group, $H$ is a subgroup of $G$ and $B$ is a right transversal of $H$ in $G$. If $B$ is a twisted subgroup of $G$, then $B$ with the operation $*$ defined by 

\begin{equation}
\label{eqprop1}
b*b' = c, \textrm{ where } bb' = hc, \textrm{ for some } h\in H,
\end{equation}

is a Bol loop. Conversely, if $(B,*)$ is a Bol loop and $H$ is core-free, then $B$ is a twisted subgroup of $G$.
\end{prop}

If $L$ is a loop, the triple $(G,H,B)$, where $G = Mlt_r(L)$, $H = Inn_r(L)$ and $B = \{R_x\,|\,x\in L\}$, is a Baer triple. In this condition, $L$ is a Bol loop if and only if $B$ is a twisted subgroup of $G$ \cite[6.1]{Ach}.

Let $B$ be a Bol loop of exponent $n$ and $X$ be a subset of $B$. We say that $X$ is a \emph{free set of generators} of $B$ if $X$ genetares $B$ and every mapping between $X$ and a Bol loop $B'$ of exponent $n$ can be extended to a homomorphism between $B$ and $B'$. We say that $B$ is a \emph{free Bol loop of exponent $n$} if it has a free set of generators.

Now consider $B$ as a free Bol loop of exponent two. A subset $T\subset B$ is a \emph{prebasis} of $B$ if for every $b\in B$ there exist $b_1,...,b_n\in T$ such that $b=b_1b_2...b_nb_{n-1}...b_2b_1.$ Here and in the following, we will write $v=v_1v_2....v_n$ if $v=(...((v_1v_2)v_3)...)v_n.$ A subset $T\subset B$ is an \emph{independent}  if for every  $a_1,...,a_m,b_1,...,b_n\in T,$ such that $b_i\not=b_j$ and $a_p\not=a_q,$ for all $i,j,p,q,$ from  $a_1a_2...a_ma_{m-1}...a_2a_1 = b_1b_2...b_nb_{n-1}...b_2b_1$, we have that $n=m$ and $a_i=b_i,$ $i=1,...,n.$ A subset $T\subset B$ is a \emph{basis} of $B$ if $T$ is an independent prebasis of $B$.

A group $G$ is a \emph{free $2$-group} if it is a free product of cyclic groups of order two, i.e., it has the form $G=\prod_{x\in T}\star <x| x^2=1>$.

\section{Construction of a basis of free Bol loops of exponent two}

Let
$X$ be a finite ordered set of letters and 
$P=P(X)$ be the set of all non-associative words on $X.$ We denote the empty word by $1$.
For $v\in P$, by $Sub(v)$ we denote the set of all subwords of $v.$ Note that if 
$v=v_1v_2,$ then $Sub(v)=\{v\}\cup Sub(v_1)\cup Sub(v_2).$ 

For $v\in P$, the lenght of $v$, denoted by $|v|$, is the number of letters in the word $v$. Note that $|1| = 0$.

Let $C(X) = \{uu,(uv)v\,|\, u,v\in P\}$ and $W=W(X) = \{v\in P\,|\, Sub(v)\cap C(X) = \emptyset\}$. Define the mapping $\pi: P \to W$, where, for $v\in P$, $\pi(v)$ is given by induction on $|v|$ using the following rules:

(i) $\pi(x) = x$, if $x\in X$,
\\
(ii) If $u,v\in W$, then

\begin{center}
$\pi(uv) = \left\{\begin{array}{rl}
1, & \textrm{if } u = v, \\
a, & \textrm{if } u = av,\\
uv, & \textrm{if } uv\in W,
\end{array}\right.$
\end{center}

(iii) If  $u\not\in W$ or $v\not\in W$, then
 $\pi(uv) = \pi(\pi(u)\pi(v))$. 
 
 Notice that in the case (iii) we get
 $|\pi(u)\pi(v)|<|uv|$. Hence this definition is correct.

\begin{lem} \label{lema1d} Let $u,v,w,v_1,...,v_n\in P$ and $a \in W$. Then:

(a) $\pi(uv) = \pi(\pi(u)\pi(v))$.\\
(b) $\pi(uv.v)=\pi(u.vv) = \pi(u)$.\\
(c) $\pi(u) = \pi(v)$ if and only if $\pi(uw) = \pi(vw)$.\\
(d) If $\pi(u v_1v_2...v_n) = \pi(v)$, then $\pi(u) = \pi(v v_n...v_2v_1)$.\\
(e) If $\pi(v_1v_2...v_n) = a$, then $\pi(a v_n...v_2v_1) = 1$.\\
(f) If $\pi(uv) = \pi(uw)$, then $\pi(v) = \pi(w)$.
\end{lem}
\begin{proof} The item (a) follows from the definition of $\pi$. The item (c) is a consequence of (a) and (b), and the items (d) and (e) are consequences of (b) and (c). Let us prove (b) and (f).

(b) By (a), we have $\pi(uv.v) = \pi(\pi(\pi(u)\pi(v))\pi(v))$ and $\pi(u.vv) = \pi(u)$. If $\pi(u)\pi(v) \in W$, then $\pi(\pi(\pi(u)\pi(v))\pi(v)) = \pi((\pi(u)\pi(v))\pi(v)) = \pi(u)$. If $\pi(u) = c\pi(v)$, then $\pi(\pi(\pi(u)\pi(v))\pi(v)) =  \pi(c\pi(v)) = \pi(u)$.

(f) By (a), we only have to prove the case where $u,v,w\in W$.  If either $1\in \{u,v\}$ or $u = v$, the result is trivial. Suppose that $u,v\in W\setminus \{1\}$ and $u\not = v$. If $u = cv$, for some $c\not = 1$, then $\pi(uw) =  c$. Since $|c|<|u|$, we have $uw\not \in W$. Thus $u = dw$, for some $d\not = 1$, and we have $cv = u = dw$. Therefore $w = v$.

Now suppose that $uv \in W$. Since $|uv|>|u|$, it follows that $uw\in W$. Hence $uw = uv$ and we have $w = v$.
\end{proof}

\begin{lem}\label{0}
Let $v=v_1v_2v_3...v_m,$ where $v_i\in W$ and $v_i\not=v_{i+1},$ for
$i=1,...,m-1.$ If $|\pi(v)|<|v|,$ then there are three possibilities:

(a) $v_1v_2\in W$ and $v_i=v_1v_2v_3...v_{i-1}$, for some $i>2$,
\\
(b) There exists $v^{\prime}_{1}\in W$ such that $v_{1}=v^{\prime}_{1}v_{j}v_{j-1}...v_{ 3}v_{2},$ where $1<j<m$ and $ v^{\prime}_{1}v_{j+1}\in W,$
\\
(c) $v_1 = v^{\prime}_{1} v_mv_{m-1}...v_2$, for some $v^{\prime}_{1}\in W$.

\end{lem}
\begin{proof} If $v_1v_2\in W$, then there exists $i\in \{3,...,m\}$ such that $v_1v_2...v_{i-1}\in W$ and $v_1v_2...v_{i-1}v_i\not \in W$. Since $v_i\not=v_{i+1},$ we have $v_i=v_1v_2v_3...v_{i-1}$. When $v_1v_2\not \in W$ we have that $v_1 = \alpha v_2$, for some $\alpha \in W$. If $v_1 \not = \beta v_mv_{m-1}...v_2$, for every $\beta \in W$, then there exist $v^{\prime}_{1}\in W$ and $j\in \{2,3,...,m-1\}$ such that $v_{1}=v^{\prime}_{1}v_{j}v_{j-1}...v_{ 3}v_{2}$ and $v^{\prime}_{1}\not = \gamma v_{j+1}$, for every $\gamma \in W$. Hence $ v^{\prime}_{1}v_{j+1}\in W.$
\end{proof}

{\bf Remark.} In the Lemma \ref{0} it is possible that $v_1^{\prime}=1.$\\

The following result is a consequence of Lemma \ref{0}.

\begin{cor}\label{00} Let $v=v_1v_2v_3...v_m,$ where $v_i\in W$ and $v_i\not=v_{i+1},$ for
$i=1,...,m-1.$ There are four possibilities:

(a) $\pi(v)=1$,

(b) $\pi(v)=v_lv_{l+1}...v_m$, where $\pi(v_1v_2...v_{l-1})=1$ and $1\leq l \leq m$,

(c) $\pi(v)=v_{l}^{\prime}v_{j+1}v_{j+2}...v_m$, where $\pi(v_1v_2...v_{l-1})=1, v_l = v_{l}^{\prime}v_{j}v_{j-1}...v_{l+1}$, $v_{l}^{\prime}\not = 1$ and $1\leq l < j < m$,

(d) $\pi(v)=v_{l}^{\prime}$, where $\pi(v_1v_2...v_{l-1})=1, v_l = v_{l}^{\prime}v_{m}v_{m-1}...v_{l+1}$, $v_{l}^{\prime}\not = 1$ and $1\leq l < m$.
\end{cor}

Consider $X = \{x_1,x_2,...,x_r\}$.  We define an order $>$ in $W$ inductively by the following rules:

(i) $x_i > x_j$, if $i>j$,
\\
(ii) $u>v$, if $|u|>|v|$,
\\
(iii) If $|u|=|v|$, $u = u_1u_2$, $v =v_1v_2$, then $u>v$ in the following cases:
\\
(iii.1) $u_2 > v_2$, 
\\
(iii.2) $u_2 = v_2$ and $u_1 > v_1$.

\begin{defi}\label{g1}
For any $y\in P$ there exists unique canonical decomposition
$y=y_1y_2...y_{m-1}y^{\prime}_m$ such that $|y_1|=1.$ We denote
$y^t=y^{\prime}_my_{m-1}...y_1.$ If 
$y^{\prime}_m=y_ky_{k-1}...y_m$ with $|y_k|=1,$ then  $(y^{t})^{t}=y^{tt}=y_1y_2...y_m....y_k$ and $y^{ttt}=y_ky_{k-1}...y_m....y_1=y^t.$
\end{defi}

\begin{defi}\label{g2} In notation above, define the following:

(i) $||y||=m,$

(ii) $y^*=\{x\in P| x^{tt}=y^{t},\,or\,\,x^{tt}=y^{tt}\}.$

(iii) $y_{(i)}=y_ky_{k-1}...y_i(y_1y_2...y_{i-1}), i=3,...,k,$ 

(iv) $ y^{(i)}=y_1y_{2}...y_{i-1}(y_ky_{k-1}...y_{i}), i=2,...,k-1,$
\end{defi}

\begin{ex}
\label{ex1} Let $X = \{a,b,c\}$ and $y=(a(bc))((ca)b).$ Then the canonical decomposition of $y$ is $y = y_1y_2y_3'$, where $y_1=a,$ $y_2=bc,$ $y_3'=(ca)b = y_5y_4y_3$, and hence $||y||=3$ and $||y^t||=||y^{tt}||=5.$ Furthermore, $y^t=(((ca)b)(bc))a$, $y^{tt} = (((a(bc))b)a)c$, $y_{(3)}=y_5y_4y_3(y_1y_2),$ $y_{(4)}=y_5y_4(y_1y_2y_3),$ $y_{(5)}=y_5(y_1y_2y_3y_4),$ $y^{(2)}=y_1(y_5y_4y_3y_2),$ $y^{(3)}=(y_1y_2)(y_5y_4y_3),$ and $y^{(4)}=(y_1y_2y_3)(y_5y_4)$. Note that $y=y^{(3)}$ and \\$y^*=\{y^t,y^{tt}, y_{(3)}, y_{(4)}, y_{(5)},y^{(2)},y^{(3)}, y^{(4)}\}.$
 \end{ex}

Define the set of symmetric words of $P$ by $S(X) = \{y_1y_2...y_my_{m+1}y_m...y_1\,\,|\,\, y_i\in P ,m > 0\} $.

\begin{lem}\label{l1} In notation above, we have:

(a)  $y^*=\{y^{tt}=y^{tttt},y^{t},y_{(i+1)}, y^{(i)}, i=2,...,k-1\}$ and $|\{y^{tt},y^{(i)}| i=2,...,k-1\}| = |\{y^{t},y_{(i)}| i=3,...,k\}| = k-1$.

(b) If $y^{t}=y^{tt},$
then 
$y^*=\{y^{tt}, y^{(i)}=y_{(k-i+2)}, i=2,...,k-1\}$ and
$|y^*|=k-1$.

(c) If $y^{t}\not=y^{tt},$ then $|y^*|=2(k-1)$ and $y^*\cap S(X)=\emptyset.$

(d) If $y^{t},y^{tt}\in W,$ then $y^*\subset W$.

(e) $min\{y^{t}, 
y^{tt}\}=min\{x | x\in y^*\}.$

\end{lem}
\begin{proof}
(a) It is immediate that $\{y^{t},y^{tt},y_{(i+1)}, y^{(i)}, i=2,...,k-1\}\subset y^*$ and $|\{y^{tt},y^{(i)}| i=2,...,k-1\}| = |\{y^{t},y_{(i)}| i=3,...,k\}| = k-1$.

Let $z\in y^*$. We have that $z^{tt} = z_1z_2...z_r$, where $z = z_1z_2...z'_l$, $z'_l = z_rz_{r-1}...z_l$ and $|z_1| = |z_r| = 1$. Since $|z_1| = |z_r| = 1$ and $z^{tt} \in \{y^{t},y^{tt}\}$, we have $k = r$. If $z^{tt} = y^{tt}$, then $z_i = y_i$, for all $i$, and so $z\in \{y^{tt},y^{(i)}| i=2,...,k-1\}$. If $z^{tt} = y^{t}$, then $z_{k+1-i} = y_i$, for all $i$, and so $z\in \{y^{t},y_{(i)}| i=3,...,k\}$. Therefore $y^*=\{y^{t},y^{tt},y_{(i+1)}, y^{(i)}, i=2,...,k-1\}$.

(b) If $y^{t}=y^{tt},$ then $y^{(i)}=y_{(k-i+2)}$, for all $i\in \{2,...,k-1\}$. Thus the claim follows from (a).

(c) If $y_{(i)} = y^{(j)}$, for some $i$ and $j$, then a simple calculation shows that $j = k-i+2$ and $y_l = y_{k+1-l}$, for all $l$, and so $y^{t}=y^{tt}$. Hence $|y^*|=2(k-1)$ by (a). By a similar argument, we can get that $y^{(i)},y_{(i+1)} \not \in S(X)$, for all $i$.

(d) Let $y^{(i)} = y_1y_{2}...y_{i-1}(y_ky_{k-1}...y_{i}) \in y^*$. Since $y^{t},y^{tt}\in W,$ we have $y_1y_{2}...y_{i-1},y_ky_{k-1}...y_{i} \in W$, $y_{i-1}\not = y_i$ and $y_{i-1}\not = y_ky_{k-1}...y_{i}$. Then $y_1y_{2}...y_{i-1}\not = \alpha (y_ky_{k-1}...y_{i})$, for all $\alpha \in P$. Hence $y^{(i)}\in W$. By similar arguments, we can conclude that $y_{(j)}\in W$, for all $j$. Therefore, $y^*\subset W$.

(e) It is clear that $y^{tt} = min\{y^{tt},y^{(i)}| i=2,...,k-1\}$ and $y^{t} = min\{y^{t},y_{(i)}| i=3,...,k\}$. Thus the claim follows from (a).
\end{proof}

\textbf{Remark.} We can define an equivalence relation $\sim$ on $P(X)$ by $x\sim y$ if and only if $x^* \cap y^* \not = \emptyset$. The equivalence classes of this relation can be of three types: $O_1$, $O_2$ and $O_3$, where:

(i) $O_1\subset W(X)\setminus S(X)$,

(ii) $O_2\subset W(X)$ and $y^t = y^{tt}\in S(X)$, for $y\in O_2$,

(iii) $O_3\not \subset W(X)$.

\begin{defi}\label{d1} For $y\in W$, let $y_0=min\{y^{t}, 
y^{tt}\}$. Define the set $D=D(X)=\{y_0 | y,y_0,y_0^t\in W, y_0^t\not=y_0 \}$. 
\end{defi}

\begin{ex}
\label{ex2} If $X = \{a,b\}$ with $b>a$ and $W_n = \{y\in W\,|\,|y| = n\}$, then 

$D\cap W_5=\{a,b,ba,((ba)b)a,(b(ab))a,(b(ba))a,((ba)(ab))a,((a(ba))b)a,((b(ab))a)b,$

$((b(ba))b)a,(b(a(ab)))a,(b(a(ba)))a,(b(b(ab)))a,(b(b(ba)))a,(b((ab)a))a,(b((ba)b))a\}.$
\end{ex}

\begin{defi}\label{d2} Define the following sets:

(i) $R_1 = X=\{x_1,x_2,...,x_r\}$,

(ii) $R_n = R_{n-1} \cup \{y\in D(X)\,\,|\,\,|y|\leq n, y = u_1u_2...u_m, u_i\in R_{n-1}, i=1,...,m\}$, for $n>1$, 

(iii) $R(X) = \displaystyle \bigcup_{n\in \mathbb{N}} R_n$.

\end{defi}

Notice that $X\subset R(X)\subset W$  and $R(X)\cap S(X) = \emptyset$.

\begin{cor}\label{l111}  Let $b=b_1b_2...b_n\in W,$ be such that $b_1\in X$. If $b\in R(X)$, then 

 \begin{equation}\label{meq}
 b<b^{t}, b^*\subset W, b_n\in X \textrm{ and } b_i\in R(X),\textrm{ for } i=1,...,n.
 \end{equation}
\end{cor}

\begin{ex}
\label{ex3} If $X = \{a,b\}$ with $b>a$, then: 

$R_5=\{a,b,ba,((ba)b)a,(b(ba))a,((a(ba))b)a, ((b(ba))b)a\}.$

Note that $(b((ba)b))a,(b((ab)a))a\in (D\cap W_5)\setminus R_5,$
since $(ba)b,(ab)a\in S(X)$.

\end{ex}

\begin{defi}\label{d3}
$B(X)=\{1\}\cup \{y\in W(X) | y=y_1y_2...y_n, y_i\in R(X)\}.$
\end{defi}

{\bf Remark.} Let $y= y_1y_2...y_n\in P$ be such that $y_i\in R(X)$, for all $i$. By Lemma \ref{0}, $y\in W$ if and only if $y_1y_2...y_{i-1}\not= y_i\not=y_{i+1},$ for $i\in\{1,...,n-1\}.$

\section{Proof that $R(X)$ is a basis of $B(X)$. }

For proof that $R(X)$ is a basis of $B(X)$ we need the detailed
information about $\pi(b)$ if $b=b_1...b_k...b_1$, $b_i\not=b_{i+1}$
and $b_i\in R(X).$ We begin with the following simple fact.
 
\begin{lem}
\label{prop1d} Let $b_1,b_2,...b_k\in P$. Then $\pi(b_1b_2...b_kb_{k-1}...b_1) = 1$ if and only if $\pi(b_k) = 1$.
\end{lem}
\begin{proof} We have $\pi(b_1...b_k...b_1) = \pi(\pi(b_1)...\pi(b_k)...\pi(b_1))$. If $\pi(b_k) = 1$, then it is clear that $\pi(b_1...b_k...b_1) = 1$.

Now suppose that $\pi(b_1...b_k...b_1) =1$. Omitting all $b_j,b_{j+1}$ such that $\pi(b_j) = \pi(b_{j+1})$, we get that $\pi(\pi(b_1)...\pi(b_k)...\pi(b_1)) = \pi(a_1a_2...a_ra_{r-1}...a_1)$, where $r\leq k$, $a_r = \pi(b_k)$, $a_i\not = a_{i+1}$ and $a_i\in W\setminus \{1\}$, for all $i<r$. 

We will prove that $a_r = 1$ by induction on $r$. Consider $r>1$ and define $a_{r+i} = a_{r-i}$, for all $i$. Let $l$ be the minimal such that $\pi(a_1a_2...a_l) = 1$. If $l<2r-1$, then $\pi(a_{l'}a_{l'-1}...a_1) = 1$, where $l' = 2r-1-l$, and so $\pi(a_1a_2...a_{l'}) = 1$ by Lemma \ref{lema1d}. Thus we only have to consider three cases:

(i) $l<r$. Then $\pi(a_{l+1}...a_r...a_{l+1}) = 1$, and hence $a_r = 1$ by the induction hypothesis.

(ii) $l=r$. Then $\pi(a_1a_2...a_r) = \pi(a_1a_2...a_{r-1}) = 1$, and we get $a_r = 1$.

(iii) $l = 2r-1$. By Lemma \ref{0}, if $a_r\not = 1$, then either $a_1 = a_1a_2...a_ra_{r-1}...a_2$ or $a_1 = va_sa_{s-1}...a_2$, for some $v\not = 1$ and $s>0$ such that $2(s-1) = 2r-3$, but both cases are impossible. Hence $a_r=1$.
\end{proof}

\vspace{4mm}

\begin{lem}
\label{lema1a} Let $n>1$ and $c,w_1,w_2,...,w_n \in W\setminus \{1\}$ be such that $cw_1 \in W$, $w_i \in Sub(c)\cup Sub(w_1)$ and $w_{i-1}\not = w_i$, for all $i$. Then $cw_1w_2...w_n \in W$.
\end{lem}
\begin{proof}
For $1\leq m < n$, suppose that $cw_1w_2...w_m \in W$. Since $w_{m+1} \in Sub(c)\cup Sub(w_1)$, we have that $w_{m+1}\not = cw_1w_2...w_m$. Since $w_m \not = w_{m+1}$, there is no $\beta$ such that $cw_1w_2...w_m = \beta w_{m+1}$. Hence $cw_1w_2...w_mw_{m+1} \in W$.
\end{proof}

\vspace{4mm}

\begin{lem}
\label{lemma2} Let $k>1$ and $w = w_1w_2...w_kw_{k-1}...w_1 \in S$ be such that $w_i\in W\setminus \{1\}$,  $w_1w_2\in W$ and $w_i\not = w_{i+1}$ for all $i$. There are two possibilities:

(a) $\pi(w) = w$ or\\
(b) There exists $l$ such that $3\leq l \leq k$ and $w_l = w_1w_2...w_{l-1}$.

\end{lem}
\begin{proof} If $k = 2$, then $w_1w_2w_1\in W$ since $w_1w_2\not = \alpha w_1$, for all $\alpha \in W$. Hence $\pi(w) = w$. 

Suppose that $k\geq 3$ and $\pi(w)\not = w$, and define $w_{k+i} = w_{k-i}$, for all $i$. By Lemma \ref{0} (a), there exists $l$ such that $2<l\leq 2k-1$ and $w_l = w_1w_2...w_{l-1}$. Since $w_l$ is not a proper subword of itself, we must have $l\leq k$.
\end{proof}


\begin{prop}
\label{lema6} Let $b = b_1b_2...b_kb_{k-1}...b_1 \in S$ be such that $b_1 \in W\setminus \{1\}$, $b_i\in R$ and $b_{i-1}\not = b_i$ for all $i>1$. Then  $\pi(b) = \lambda b_1$, where $\lambda = 1$ implies that $k = 1$ or $b_1\not\in R$.

Moreover, if $b_1\in R$, then $\pi(b) \in R$ if and only if $k = 1$.
\end{prop}
\begin{proof}

If $k \in \{1,2\}$ it is easy to see that the claim holds. Suppose that the claim holds for all $k'<k$, where $k\geq 3$. First we will prove the   following lemmas.\\

\begin{lem}
\label{le7} Suppose that $b_m = b_1b_2...b_{m-1}$, where $3\leq m \leq k$. Then $\pi(b) = \epsilon b_{m-1}...b_2b_1\not\in R$.
\end{lem}

\begin{proof} We have three cases:

(i) $m = k$. Thus $b_k = b_1b_2...b_{k-1}$, and hence $\pi(b) = \pi(b_{k-1}...b_2b_1)$. Since $b_{k-1}...b_2b_1 \in b_k^*$, it follows that $\pi(b) = b_{k-1}...b_2b_1$. Since $b_k \in R$ and $R\cap S=\emptyset$, we have $\pi(b)=b_{k-1}...b_2b_1 \not \in R$.\\

(ii) $m = k-1$. Thus $b_{k-1} = b_1b_2...b_{k-2}$ and $\pi(b) = \pi(b_kb_{k-1}...b_1)$. Since $|b_{k-1}|>1$ and $b_k\in R$, we have $b_kb_{k-1} \in W$ by \eqref{meq}. By Lemma \ref{lema1a}, $b_kb_{k-1}...b_1 \in W$, and then $\pi(b) = b_kb_{k-1}...b_1$. Since $b_1b_2...b_{k-1}\not \in W$, it follows that $b_1b_2...b_k \in \pi(b)^*\setminus W$, and hence $\pi(b)  \not \in R$ by \eqref{meq}.\\

(iii) $m<k-1$. Thus $\pi(b) = 
 \pi(b_{m+1}...b_kb_{k-1}...b_{m+1}...b_2b_1)$. By the induction hypothesis, $\pi(b_{m+1}...b_kb_{k-1}...b_{m+1}) = \lambda b_{m+1}$, where $\lambda \not = 1$ because $m+1 < k$ and $b_{m+1} \in R$. Then $\pi(b) = \pi(b_{m+1}...b_kb_{k-1}...b_{m+1}...b_2b_1) = \pi(\lambda b_{m+1}...b_2b_1)$. If  $b_m = \lambda b_{m+1}$, then $\pi(b) = \pi(b_{m-1}...b_2b_1)$. Since $b_{m-1}...b_2b_1 \in b_m^*$, it follows that $\pi(b) = b_{m-1}...b_2b_1$. Furthermore, since $b_m\in R$ and $b_{m-1}...b_2b_1 \not = b_m$, we have $\pi(b) = b_{m-1}...b_2b_1 \not \in R$.

When $\lambda b_{m+1}b_m \in W$ we have that  $\lambda b_{m+1}b_m ...b_2b_1 \in W$ by Lemma \ref{lema1a}. Then $\pi(b) = \lambda b_{m+1}...b_2b_1$. Since $b_1b_2...b_m\not \in W$, we have $b_1b_2...b_{m+1}\lambda \in \pi(b)^*\setminus W$, and hence $\pi(b)\not \in R$ by \eqref{meq}.

Therefore, we proved Lemma \ref{le7}.
\end{proof}
\vspace{4mm}

Define $b_{k+i} = b_{k-i}$, for all $i$. Note that $b = b_{2k-1}b_{2k-2}...b_1 = b_1b_2...b_{2k-1}$.\\

\begin{lem}
\label{lema8} 
 Suppose that for $n\in \{2,3,...,k\}$, $b_1\in R$ and we have one of the following situations:

(a$1$) $b_n = b_n' b_1b_2...b_{n-1}$, where $b_n'\not = 1$, or

(a$2$) $b_n = b_2...b_{n-1}$.

Then  $\pi(b_{2k-1}b_{2k-2}...b_{n+1})\not = 1$.
\end{lem}

\begin{proof} 
  Suppose by contradiction that $\pi(b_{2k-1}b_{2k-2}...b_{n+1}) = 1$. We have two cases:\\

(i) $n< k$. Let $v = b_1b_2...b_n$. Then $\pi(vb_{n+1}...b_k...b_{n+1})  = 1$, and so $\pi(v) = \pi(b_{n+1}...b_k...b_{n+1})$ by Lemma \ref{lema1d}. By the induction hypothesis, we get that $\pi(v) = \lambda b_{n+1}$, where $\lambda = 1$ if and only if $n + 1 = k$. Applying Corollary \ref{00} to the word  $v = b_1b_2...b_n$, we have two cases $\pi(v)=\alpha b_n$ (in the cases (b) and (c)) or
$\pi(v)=b_l^{\prime}$, 
$\pi(b_1...b_{l-1})=1,$
 $b_l=b_l^{\prime}b_nb_{n-1}...b_{l+1}$ (case (d)). We note that the case (a) is impossible since $\pi(v)\not=1.$

Let $\pi(v) = \alpha b_n$. If $\alpha = 1$, then $\lambda b_{n+1} \in R$. Since $b_1\in R$, we get that $n+1 = k$ by the induction hypothesis, and hence $b_n = \pi(v) = b_{n+1}$, which is a contradiction. Suppose that $\alpha \not = 1$. Since $b_n \not = b_{n+1}$, it follows that $\lambda = 1$, and then $|b_n| = 1$ by \eqref{meq}, which contradicts (a$1$) and (a$2$). 

 Let $\pi(v)=b_l^{\prime}$ and $ l=1.$ Then $b_1 = \lambda b_{n+1} b_n b_{n-1}...b_2$. In (a1) this does not occur since $b_1 \in Sub(b_n)$. Now consider the case (a2). Since $\pi(b_2b_3...b_n) = 1$ in this case, we get that $b_1^t \in b_1^*\setminus W$, and then $b_1\not \in R$ by \eqref{meq}, a contradiction.

If $\pi(v)=b_l^{\prime}$ and $1<l<n$, $\pi(b_1b_2...b_{l-1}) = 1$ and $b_l = b_l^{\prime}b_n b_{n-1}...b_{l+1}$. Then we have a contradiction since $b_l \in Sub(b_n)$ in both cases (a1) and (a2).\\

(ii) $n = k$. By assumption, we have that $\pi(b_1b_2...b_{k-1})  = 1$, and then $\pi(b_{k-1}...b_2b_1)  = 1$ by Lemma \ref{lema1d}. First, consider the case (a1). Since $b_k = b_k' b_1b_2...b_{k-1}$ and $\pi(b_{k-1}...b_2b_1)  = 1$, it follows that $b_k^t \in b_k^*\setminus W$, and then $b_k\not \in R$ by \eqref{meq}, a contradiction.

Now consider the case (a2). Since $b_k = b_2...b_{k-1}$, then $|b_{k-1}| = 1$ by \eqref{meq}, and so $b_{k-1}b_{k-2}\in W$. Since  $\pi(b_{k-1}...b_2b_1)  = 1$, it follows that there exists $l$ such that $b_l = b_{k-1}b_{k-2}...b_{l+1}$ by Lemma \ref{0}. If $l>1$, then $\pi(b_{k-1}b_{k-2}...b_{l+1}b_l) = 1$, and so $b_k^t \in b_k^*\setminus W$, which is a contradiction. If $l = 1$, then $b_1 = b_k^t$, which is a contradiction since $b_1,b_k \in R$.

Therefore, Lemma \ref{lema8} is proved.
\end{proof}

\vspace{4mm}

Now we can finish the proof of Proposition \ref{lema6}. First, let us prove that $\pi(b) = \lambda b_1$, for some $\lambda \in W$, where $\lambda \not = 1$ if $b_1\in R$. By Lemma \ref{prop1d}, we have that $\pi(b)\not =1$, and then there are three possibilities according to Corollary \ref{00}:

(i) $\pi(b) = b_n'b_m...b_1$, where $b_n = b_n'b_{m+1}b_{m+2}...b_{n-1}$, $1\leq m<n \leq 2k-1$ and $b_n'\not =1$ if $m=1$. Thus we have the desired result.\\

(ii) $\pi(b) = b_n'$, where $b_n = b_n'b_1b_2...b_{n-1}$, $1<n<2k-1$, $b_n'\not = 1$ and $\pi(b_{2k-1}b_{2k-2}...b_{n+1}) = 1$. Since $b_n$ can not be a proper subword of itself, we get $n\leq k$. Furthermore, we get that $b_1\in R$ by \eqref{meq}. Then we have a contradiction with Lemma \ref{lema8}.\\

(iii) $\pi(b) = b_nb_{n-1}...b_1$, where $\pi(b_{2k-1}b_{2k-2}...b_{m+1}) = 1$, $b_m = b_{n+1}b_{n+2}...b_{m-1}$ and $1\leq n<m\leq 2k-1$. If $n>1$ or $n=1$ and $b_1 \not \in R$, then we have the desired result. Suppose that $\pi(b) = b_1\in R$. We have two cases:\\

(iii.1)  $|b_2|>1$. By \eqref{meq}, $b_1$ can not be of the form $b_1^{\prime} b_2$, and then $b_1b_2\in W$. By Lemmas \ref{lemma2} and \ref{le7}, we get that $\pi(b) \not \in R$, which is a contradiction.
\\

(iii.2)  $|b_2|=1.$ Note that $b_m =  b_2b_3...b_{m-1}$. Since $b_m$ can not be a proper subword of itself and $b_2b_3...b_k...b_2\not \in R$, it follows that $m\leq k$. Then $\pi(b_{2k-1}b_{2k-2}...b_{m+1}) \not = 1$ by Lemma \ref{lema8}, which is a contradiction.\\

Now we only have to prove that $\pi(b) \not \in R$ when $b_1\in R$. Consider that $b_1\in R$ and $\pi(b) = \lambda b_1$, where $\lambda \not = 1$. If $|b_1| > 1$, then $\pi(b)\not \in R$ by \eqref{meq}. If $|b_1| = 1$, then $b_1b_2\in W$, and as in (iii.1) we get that $\pi(b) \not \in R$.
\end{proof}

\begin{cor}
\label{cor0}
Let $b = b_1b_2...b_kb_{k-1}...b_1 \in S$ and $g = g_1g_2...g_ng_{n-1}...g_1 \in S$ be such that $b_i,g_j \in R$, $b_{i-1}\not = b_i$ and $g_{j-1}\not = g_j$, for all $i$ and $j$. If $\pi(b) = \pi(g)$, then $b_1 = g_1$.
\end{cor}

As a consequence of Proposition \ref{lema6} and  Lemma \ref{le7} we have the following result:

\begin{cor}
\label{cor1}
Let $b = vb_1b_2...b_kb_{k-1}...b_1v \in S$ be such that $v \in W\setminus \{1\}$, $vb_1 \in W$, $b_i\in R$ and $b_{i-1}\not = b_i$ for all $i>1$. There are two possibilities:

(a) $\pi(b) = b$,\\
(b) There exists $m$ such that $2\leq m \leq k$ and $\pi(b) = \epsilon b_{m-1}...b_1v$, where $\epsilon \in W$.

\end{cor}

\begin{lem}
\label{lema7b} Let $b = vb_1v \in S$ and $g = vg_1g_2...g_ng_{n-1}...g_1v \in S$ be such that $v \in W\setminus\{1\}$, $v\not = b_1$, $v\not = g_1$, $b_1,g_j \in R$, and $g_{j-1}\not = g_j$, for all $j$. If $\pi(b) = \pi(g)$, then $b = g$.
\end{lem}
\begin{proof}
If $n = 1$, then $\pi(vb_1) = \pi(vg_1)$. By Lemma \ref{lema1d} (f), $b_1 = g_1$, and hence $b = g$. Now suppose that $n>1$ and the claim holds for every $n'<n$. We will prove this result in two steps:

(i) First we will prove that there exists $\alpha \in W\setminus \{1\}$ such that $\alpha b_1 \alpha\in W$, \\$\pi(\alpha b_1 g_1g_2...g_ng_{n-1}...g_1 b_1\alpha) = \alpha b_1\alpha$ and either $\alpha b_1\not = g_1$ or $\alpha=v$. We have two cases:

(i.1) $v = \alpha b_1$, with $\alpha \not = 1$. Then $\pi(vb_1v) = \pi(\alpha v) = \pi(vg_1g_2...g_ng_{n-1}...g_1v)$, and hence by Lemma \ref{lema1d} (c):

\begin{equation}
\label{eqlema7b}
\pi(\alpha) = \pi(vg_1g_2...g_ng_{n-1}...g_1).
\end{equation}

Using Lemma \ref{lema1d} (c) and $v = \alpha b_1$ in \eqref{eqlema7b}, we get $\pi(\alpha b_1\alpha) = \pi(\alpha b_1 g_1...g_n...g_1 b_1\alpha)$.\\ 

(i.2) $vb_1 v \in W$. By $\pi(vg_1g_2...g_ng_{n-1}...g_1v) = \pi(vb_1v) = vb_1v$ and Lemma \ref{lema1d} (e), we get $\pi(vb_1 g_1g_2...g_ng_{n-1}...g_1  v) = 1$. Thus 

\begin{center}
$\pi(vb_1 g_1...g_n...g_1 b_1v) = \pi(vb_1 g_1...g_n...g_1vvb_1v) = \pi(\pi(vb_1 g_1...g_n...g_1v)\pi(vb_1v)) = vb_1v$
\end{center}

and we put $\alpha = v$.\\


(ii) Now consider $\alpha \in W\setminus \{1\}$ as in (i). If $b_1 = g_1$, then 

\begin{center}
$ \pi(\alpha b_1g_1g_2...g_n...g_1b_1\alpha) = \pi(\alpha g_2...g_n...g_2\alpha) =  \alpha b_1\alpha$.
\end{center}

If $\alpha \not = g_2$, then by induction $b_1 = g_2...g_n...g_2$. Since $b_1\in R$ and $R\cap S = \emptyset$, then $n=2$ and $b_1 = g_2$, which is a contradiction with $b_1=g_1\not = g_2$. In the case $\alpha = g_2$ and $n>2$, we get $\pi(g_3...g_n...g_3) = \alpha b_1\alpha = g_2b_1g_2$. By Corollary \ref{cor0}, we have $g_3 = g_2$, a contradiction. Finally, if $\alpha = g_2$ and $n=2$, we have $\pi(g_2) = g_2b_1g_2$. Since $g_2\in R$, hence $b_1 = g_2$, which is a contradiction with $b_1=g_1$.

Suppose $b_1 \not = g_1$. By the choice of $\alpha$, either $\alpha b_1g_1 \in W$ or $g_1 = \alpha b_1$. We have two cases: 

(ii.1) $g_1 = \alpha b_1$. Then $\alpha b_1\alpha =  \pi(\alpha b_1 g_1...g_ng_{n-1}...g_1b_1\alpha) = \pi(g_2...g_ng_{n-1}...g_1b_1\alpha)$. Hence by Lemma \ref{lema1d} (c) we get $\pi(g_2...g_ng_{n-1}...g_1) = \alpha$ and $\alpha g_1 = \alpha b_1\alpha\in W$. Using the same lemma again, we get $\pi(g_2...g_ng_{n-1}...g_2) = \alpha g_1$. If $n=2$, then $g_2=\alpha g_1 \in R$, and hence $g_1\in X$, which is a contradiction with $g_1 = \alpha b_1$. Then $n>2$. By Proposition \ref{lema6}, there exists $\lambda\not =1$ such that $\alpha g_1 = \lambda g_2$, and then $g_1 = g_2$, a contradiction.\\

(ii.2) $\alpha b_1g_1 \in W$. Note that $\alpha b_1\alpha \not = \alpha b_1 g_1...g_n...g_1 b_1\alpha$. Then there exists $m$ such that $1<m\leq n$ and $g_m = \alpha b_1 g_1...g_{m-1}$ by Lemma \ref{lemma2}. We have three more cases:

(ii.2.1) $m = n$. Then $\alpha b_1\alpha = \pi(g_{n-1}...g_1 b_1\alpha)$. Since $g_{n-1}...g_1 b_1\alpha \in g_n^*$ and $g_n\in R$, it follows  that $\alpha b_1\alpha = g_{n-1}...g_1 b_1\alpha $, which is a contradiction because  $g_n^* \cap S(X) \not = \emptyset$.

(ii.2.2) $m = n-1$. Then $\alpha b_1\alpha = \pi(\alpha b_1 g_1...g_n...g_1 b_1\alpha) =  \pi(g_ng_{n-1}...g_1 b_1\alpha)$. Since $|g_{n-1}|>1$ and $g_n\in R$, we have $g_ng_{n-1} \in W$. By Lemma \ref{lema1a}, $g_ng_{n-1}...g_1 b_1\alpha \in W$. Thus $\alpha b_1\alpha = g_ng_{n-1}...g_1 b_1\alpha$, and hence $\alpha = g_ng_{n-1}...g_1$, which is a contradiction because $\alpha \in Sub(g_{n-1})$.

(ii.2.3) $m < n-1$. Then $\alpha b_1\alpha =  \pi(g_{m+1}...g_n...g_1 b_1\alpha)$. By Proposition \ref{lema6}, there exists $\lambda \not = 1$ such that $\alpha b_1\alpha = \pi(\lambda g_{m+1}g_m...g_1 b_1\alpha)$ and $\lambda g_{m+1} \in W$. If $g_m = \lambda g_{m+1}$, similarly to (ii.2.1) we get a contradiction. If $\lambda g_{m+1}g_m\in W$, similarly to (ii.2.2) we get a contradiction.
\end{proof}

\begin{lem}
\label{lema7a} Let $b = vb_1b_2...b_kb_{k-1}...b_1v \in S$ and $g = vg_1g_2...g_ng_{n-1}...g_1v \in S$ be such that $v \in W\setminus\{1\}$, $v\not = b_1$, $v\not = g_1$, $b_i,g_j \in R$, $b_{i-1}\not = b_i$ and $g_{j-1}\not = g_j$, for all $i$ and $j$. If $\pi(b) = \pi(g)$, then $b = g$.
\end{lem}
\begin{proof} We will prove this result by induction on $k$. We can consider $n\geq k>1$. First we will prove the affirmation 1:\\

\textbf{Affirmation 1.} If $b_1 = g_1$, then $b=g$.

Proof of affirmation 1:
We have that $\pi(vb_1b_2...b_kb_{k-1}...b_2)= \pi(vb_1g_2...g_ng_{n-1}...g_2)$. Then $\pi(\pi(vb_1)b_2...b_kb_{k-1}...b_2\pi(vb_1))= \pi(\pi(vb_1)g_2...g_ng_{n-1}...g_2\pi(vb_1))$. If  $b_2 \not = \pi(vb_1)$, then $g_2 \not = \pi(vb_1)$ by Proposition \ref{lema6} and Corollary \ref{cor0}, and hence the result follows by the induction hypothesis.

Suppose that $b_2  = \pi(vb_1)$. If $k=2$, then $\pi(b_2g_2...g_ng_{n-1}...g_2b_2) = b_2$. By Proposition \ref{lema6}, we must have $b_2 = g_2$ and $n =2$, and then $b = g$. Using a similar argument, we get that $k=3$ implies $b = g$. Now consider $k>3$. Then $\pi(b_3...b_kb_{k-1}...b_3)= \pi(b_2g_2g_3...g_ng_{n-1}...g_3g_2b_2)$. By $b_2\not = b_3$ and Corollary \ref{cor0}, we must have $b_2 = g_2$ and $b_3 = g_3$, and hence the desired result follows by the induction hypothesis.

Therefore, affirmation 1 is proved.\\

By Proposition \ref{lema6}, there exist $\alpha,\beta\in W$ such that $\pi(vb_1...b_k...b_1v)=\alpha v$ and $\pi(vg_1...g_n...g_1v) = \beta v$. We have ten cases depending on $\alpha$ and $\beta$:

(i) $\alpha = 1$ (or $\beta = 1$).

(ii) $\alpha b_1,\beta g_1 \in W$ and $\alpha,\beta \not = 1$.

(iii) $\alpha =\alpha_1 b_1$ and $\beta = \beta_1g_1$, where $\alpha_1\not = 1\not= \beta_1$.

(iv) $\alpha = b_1$ and $\beta = g_1$.

(v) $\alpha = b_1$, $\beta \not = 1$ and $\beta g_1 \in W$.

(vi) $\alpha=\alpha_1 b_1$ and $\beta = g_1$, where $\alpha_1 \not = 1$.

(vii) $\alpha \not = 1$, $\alpha b_1 \in W$ and $\beta = g_1$.

(viii) $\alpha =b_1$ and $\beta = \beta_1 g_1$, where $\beta_1 \not = 1$.

(ix) $\alpha =\alpha_1 b_1$ and $\beta g_1 \in W$, where $\alpha_1\not = 1 \not = \beta$.

(x) $\alpha b_1 \in W$ and $\beta = \beta_1 g_1$, where $\alpha\not = 1 \not = \beta_1$.\\

We note that if $vb_1\in W$, then $ \pi(vb_1...b_k...b_1v) = \epsilon b_1v$ by Corollary \ref{cor1}. Analogously, if $vg_1\in W$, then $ \pi(vg_1...g_n...g_1v) = \eta g_1v$. Hence in the cases (i) and (ii) we have $vb_1,vg_1\not \in W$. Since $v\in W$, then $v = v'b_1 = v''g_1$, for some $v',v''\in W$, and so $b_1 = g_1$. The same equality $b_1=g_1$ we have in the cases (iii) and (iv). Then in those cases the claim follows by affirmation 1.\\

Case (v): Since $\beta g_1\in W$, then $vg_1\not \in W$ by Corollary \ref{cor1}. Hence $v = \beta' g_1$, for some $\beta' \not = 1$. We have that $\pi(vb_1...b_k...b_1v) = b_1v$, hence by Lemma \ref{lema1d} (c) we get $\pi(vb_1b_2...b_k...b_2) = 1$ and $\pi(b_2...b_kb_{k-1}...b_2b_1v) = 1$. Using the same lemma again, we get  $\pi(b_2...b_kb_{k-1}...b_2) = \pi(vb_1)$. If $vb_1\not \in W$, then $v = \gamma b_1 = \beta' g_1$, for some $\gamma \not = 1$, and so $b_1 = g_1$. By affirmation 1, Lemma \ref{lema7a} is proved. Let $vb_1\in W$. By Proposition \ref{lema6}, we get $\pi(b_2...b_k...b_2) = \lambda b_2 = vb_1$. If $k>2$, then $\lambda \not = 1$, and hence $b_1 = b_2$, a contradiction. Then $k=2$ and $b_2 = vb_1$. From $\pi(vg_1...g_n...g_1v) = b_1v$, we have  $\pi(vg_1...g_n...g_1) = b_1$ and $\pi(vg_1...g_n...g_1b_1) = 1$ by Lemma \ref{lema1d} (c). Hence $ \pi(b_1g_1...g_n...g_1v) = 1$ and $ \pi(b_1g_1...g_n...g_1) = v$. Then $\pi(b_1g_1...g_n...g_1b_1) = vb_1 = b_2$. If $b_1 \not = g_1$, then $b_2\not \in R$ by Proposition \ref{lema6}, which is a contradiction. In the case $b_1 = g_1$ we have Lemma \ref{lema7a} by affirmation 1.\\

Case (vi): We have that $g_1 = \alpha_1 b_1$ and $\pi(vg_1g_2...g_ng_{n-1}...g_2) = 1$. Then $\pi(g_2...g_ng_{n-1}...g_2) = \pi(vg_1)$. We have two cases:

(vi.1) $n = 2$. Then $g_2 = \pi(vg_1)$. Since $g_1v =  \pi(vb_1...b_k...b_1v)$, we have that  $\pi(g_1b_1...b_k...b_1v) = 1$, and then $\pi(g_1b_1...b_k...b_1g_1) = \pi(vg_1) = g_2$, which is a contradiction with Proposition \ref{lema6}.

(vi.2) $n > 2$. By Proposition \ref{lema6}, there exists $\lambda \not = 1$ such that $\lambda g_2 = \pi(vg_1)$. Since $g_1 \not = g_2$, if follows that $vg_1\not \in W$. Then there exists $\gamma \not = 1$ such that $v = \gamma g_1$. Since $\pi(vb_1b_2...b_kb_{k-1}...b_1v)=g_1 v$, we have that $\pi(g_1 b_1b_2...b_kb_{k-1}...b_1g_1) = \gamma$. By Proposition \ref{lema6}, there exists $\lambda' \not = 1$ such that $\lambda'g_1 = \gamma$, which is a contradiction because $v =\gamma g_1 \in W$.\\

The cases (vii) and (viii) are analogous to the cases (v) and (vi), respectively.\\

Case (ix): We have that $\pi(vb_1b_2...b_kb_{k-1}...b_2\alpha_1)=1$. Then $\pi(\alpha_1 b_2...b_kb_{k-1}...b_2b_1) = v$. Since $\pi(vg_1g_2...g_ng_{n-1}...g_1v) = \alpha_1 b_1v$, it follows that $\pi(\alpha_1 b_1g_1g_2...g_ng_{n-1}...g_1) = v$, and then $\pi(\alpha_1 b_2...b_kb_{k-1}...b_2\alpha_1) = \pi(\alpha_1 b_1 g_1g_2...g_ng_{n-1}...g_1b_1\alpha_1)$.

If $b_1= g_1$, then the claim follows by affirmation 1. Suppose that $b_1\not = g_1$. By Proposition \ref{lema6} and Corollary \ref{cor0}, we have that $\alpha_1 \not = b_2$, and hence the claim follows by the induction hypothesis.\\

Case (x): By Corollary \ref{cor1}, we have $vb_1\not \in W$. Then there exists $a\not = 1$ such that $v = ab_1$. Since $\pi(vb_1b_2...b_kb_{k-1}...b_1v)= \pi(vg_1g_2...g_ng_{n-1}...g_1v)$, it follows that $\pi(ab_2...b_kb_{k-1}...b_1)=$ $ \pi(ab_1g_1g_2...g_ng_{n-1}...g_1)$. Thus $\pi(ab_2...b_kb_{k-1}...b_2a)= \pi(ab_1g_1g_2...g_ng_{n-1}...g_1b_1a)$, and the rest of the proof is analogous to (ix).
\end{proof}

As a consequence of Proposition \ref{lema6}, Corollary \ref{cor0} and  Lemma \ref{lema7a}, we have the following result.

\begin{cor}
\label{12} Let $b = b_1b_2...b_kb_{k-1}...b_1 \in S$ and $g = g_1g_2...g_ng_{n-1}...g_1 \in S$ be such that $b_i,g_j \in R$, $b_{i-1}\not = b_i$ and $g_{j-1}\not = g_j$, for all $i$ and $j$. If $\pi(b) = \pi(g)$, then $b = g$.
\end{cor}

\begin{thm}\label{123} 
For every $g\in B(X)\setminus \{1\}$ there exist unique $g_1,...,g_m\in R(X)$
such that $g =\pi(g_1g_2...g_mg_{m-1}...g_1)$ and $g_i\not=g_{i+1}, $ for $i=1,...,m-1. $

\end{thm}
\begin{proof}
Let $g=a_1a_2...a_s\in B,$ where $a_s=b_1b_2...b_l ,$ $|a_1|=|b_1|=1$, and $a_i,b_j\in R,$ for all $i$ and $j$. First we will prove by induction on $|g|$ that there exists $s(g) = g_1g_2...g_mg_{m-1}...g_1$ such that $\pi(s(g)) = g$ and $g_i\not=g_{i+1}, $ for $i=1,...,m-1.$ If $|g|=1$ or $g\in R$, then $s(g) = g$. Now suppose that $|g|>1$ and $g\not \in R$. We have four cases:

(i) $g^t,g^{tt}\in W$ and $g^t\not = g^{tt}$. Then 

\begin{center}
$s(g) = \left\lbrace \begin{array}{rl}
a_sa_{s-1}...a_2a_1g^ta_1...a_{s-1}a_s, & \textrm{ if } g^t\in R,\\
a_sb_1b_2...b_lg^{tt}b_l...b_1a_s, & \textrm{ if } g^{tt}\in R.
\end{array} \right. $
\end{center}

(ii) $g^t\in W\cap S$. Then there exist $c_1,...,c_r\in R$ such that $g^t =c_1c_2...c_rc_{r-1}...c_1$. Thus $g = \pi(a_sa_{s-1}...a_2a_1c_1c_2...c_rc_{r-1}...c_1a_1...a_{s-1}a_s)$, and it is clear that we can get $s(g)$ from this equation.

(iii) $g^t\not \in W$. Since $\pi(g)\not = 1$, we have that $\pi(g^t)\not = 1$ by Lemma \ref{lema1d}. Then $1\leq |\pi(g^t)|< |g|$. It is not difficult to see that $\pi(g^t)\in B$. By the induction hypothesis, there exist $c_1,...,c_r\in R$ such that $\pi(g^t) = \pi(c_1c_2...c_rc_{r-1}...c_1)$, and by using similar arguments as in (ii) we get $s(g)$.

(iv) $g^{tt}\not \in W$. This case is analogous to (iii).

Now we need to prove that $s(g)$ is unique. But this is a consequence of Corollary \ref{12}.
\end{proof}

\section{Main theorem}

\begin{defi}\label{d31}
In notation of Theorem \ref{123} we put for any $g\in B\setminus\{1\}:$ 

\begin{center}
$s(g) =g_1g_2...g_mg_{m-1}...g_1$.
\end{center}

Now, define a multiplication $\circ$ on the set $B(X)$ by:

(i) $x \circ 1 = 1\circ x = x$,

(ii) $x\circ y= \pi(x y_1 y_2... y_m y_{m-1}... y_1),$ where $s(y)= y_1 y_2... y_m y_{m-1}... y_1$.

Notice that the identities $x\circ x = 1$ and $(x\circ y)\circ y=x$ can be easily obtained from the definition above.
\end{defi}

\begin{thm}
\label{mt}
The set $B=B(X)$ with multiplication $\circ$ defined above is a free Bol loop of exponent $2$ with free set of generators $X.$
\end{thm}
\begin{proof} 
It is clear that $X$ generates $B$ and if $B$ is a Bol loop, then the construction of $R(X)$ and Theorem \ref{123} give us a natural way to extend a mapping between $X$ and another Bol loop $L$ of exponent two to a homomorphism from $B$ into $L$. So we only have to prove that $B$ is a Bol loop. It is possible to prove this directly, but in this case we have to consider many particular cases. We choose the other way based on the connection of Bol loops with twisted subgroups described in the Preliminaries.

Let $G=\prod_{y\in R(X)}\star <R_y | R_y^2=I_d>$ be a free $2$-group. The group $G$ acts on $B:$ $bR_y=b\circ y$ and $bI_d = b$. Then the set $H=\{g\in G | 1^g=1\}$ is a subgroup of $G$, where $1$ is the empty word of $B(X)$.

Now, let $B'=\{I_d\}\cup\{R_y\,|\, y\in R(X)\}^G$. Note that $R_yR_zR_y\in B'$, for all $y,z\in R(X)$, and then $B'$ is a twisted subgroup of $G$.

\begin{lem}
\label{lem2}
$G = HB'$
\end{lem}

\textit{Proof of Lemma $\ref{lem2}$.} Let $g=\prod_{i=1}^mR_{g_i}\in G$ and $y=g_1g_2...g_m$. Then $1^g=\pi(y)$. If $\pi(y) = 1$, then $g\in H$.

Suppose that $\pi(y) \not = 1$ and consider $s(\pi(y))= y_1y_2...y_ky_{k-1}...y_1$, where $y_i\in R(X)$. Note that $S(g)=R_{y_1}R_{y_2}... R_{y_k}R_{y_{k-1}}...R_{y_1} \in B'$. We have that $1^{S(g)}= \pi(s(\pi(y))) = \pi(y)$, and so $\pi(y)^{S(g)} = 1$. Hence $gS(g)\in H$ and $g=(gS(g))S(g) \in HB'$.

Lemma \ref{lem2} is proved.

\begin{lem}
\label{lem13}
$H\cap(B'B') =  \{I_d\}.$
\end{lem}
\textit{Proof of Lemma $\ref{lem13}$.} Let $b = R_{b_1}R_{b_2}...R_{b_m}R_{b_{m-1}}...R_{b_1} \in B'$ and $c\in B'$ be such that $bc\in H.$ By Lemma \ref{prop1d}, it follows that $ \pi(b_1...b_m...b_1)\not = 1$, and then $c\not = I_d$. Consider $c = R_{c_1}R_{c_2}...R_{c_k}R_{c_{k-1}}...R_{c_1}$. Hence 

\begin{equation}\label{e1}
(...(b_1\circ b_2)...)\circ b_m)\circ b_{m-1})...\circ b_1)\circ c_1) ...)\circ c_k) ...)\circ c_1=1.
\end{equation}

Since $(x\circ y)\circ y=x$, we get $(...(b_1\circ b_2)...)\circ b_m)\circ b_{m-1})...\circ b_1)=$ $(...(c_1 \circ c_2) ...)\circ c_k) ...)\circ c_1$. Then $\pi(b_1b_2...b_mb_{m-1}...b_1)= \pi(c_1c_2 ...c_k ...c_1).$ By Corollary \ref{12}, we get that $m=k$ and $c_i=b_i$, for all $i$.

Lemma \ref{lem13} is proved.

As a consequence of the Lemmas above, we have that $B'$ is a right transversal of $H$ in $G$, and then $(G,H,B')$ is a Baer triple. By Proposition \ref{prop1}, we get that $B'$ with the operation $*$ defined by

\begin{center}
$b*b' = c$, where $bb' = hc$, for some $h\in H$,
\end{center}

is a Bol loop of exponent two.

Now, let us prove that $(B,\circ)\cong (B',*)$. Define $\varphi:B' \to B$ by 
\begin{center}
$\varphi(R_{y_1}R_{y_2}...R_{y_m}R_{y_{m-1}}...R_{y_1}) = \pi(y_1y_2...y_my_{m-1}...y_1)$ and $\varphi(I_d) = 1$. 
\end{center}

By Lemma \ref{prop1d} and Theorem \ref{123}, we get that $\varphi$ is a bijection.

Let $b = R_{y_1}R_{y_2}...R_{y_m}R_{y_{m-1}}...R_{y_1} \in B'$ and $c = R_{z_1}R_{z_2}...R_{z_n}R_{z_{n-1}}...R_{z_1} \in B'$. Consider $y = y_1y_2...y_my_{m-1}...y_1$, $z = z_1z_2...z_nz_{n-1}...z_1$ and $u = yz_1...z_n...z_1$. Note that $\varphi(b)\circ \varphi(c) = \pi(y)\circ \pi(z) = \pi(u)$.

If $\pi(u) = 1$, then $\pi(y) = \pi(z)$. By Corollary \ref{12}, we have $y = z$. Thus $b =c$ and $\varphi(b*c)= \varphi(I_d) = 1 = \varphi(b)\circ \varphi(c)$.

If $\pi(u) \not = 1$, consider $s(\pi(u)) = u_1u_2...u_ru_{r-1}...u_1$ and $g = R_{u_1}R_{u_2}...R_{u_r}R_{u_{r-1}}...R_{u_1}$. By the proof of Lemma \ref{lem2}, we have that $b*c = g$. Hence $\varphi(b*c) = \varphi(g) = \pi(u) = \varphi(b)\circ \varphi (c)$. 

Therefore $\varphi$ is an isomorphism and we have that $(B,\circ)$ is a Bol loop of exponent two.
\end{proof}

\begin{prop}
\label{prophg} $H$ is a core-free subgroup of $G$.
\end{prop}
\begin{proof} Let $N\leq H$ be such that $N$ is normal in $G$. Suppose that $N\not = \{I_d\}$. Then there exists $\phi = R_{y_1}R_{y_2}...R_{y_n}\in N$, with $n>1$, $y_i\in R(X)$ and $y_i \not = y_{i+1}$, for all $i$. 

Since $1\phi = 1$, it follows that $\pi(y_1y_2...y_n) = 1$. Then 

\begin{equation}
\label{eqhg}
\pi(y_n...y_2y_1y_2) = y_2.
\end{equation}

Since $N \triangleleft G$ and $N\leq H$, we have $R_{y_2}\phi R_{y_2}\in H$, and then $1R_{y_2}\phi R_{y_2} = 1$. Thus $1R_{y_2}\phi  = 1R_{y_2}$, and hence $y_2\phi =y_2$. By \eqref{eqhg}, we get $\pi(y_n...y_2y_1y_2)\phi = y_2$, and then $y_2  = \pi(y_n...y_2y_1y_2y_1y_2...y_n)$, which is a contradiction with Proposition \ref{lema6}.
\end{proof}

Now we will determine the nuclei and the center of $B(X)$. Firstly, we need the following lemma.

\begin{lem}
\label{lemanuc} Let $x,z\in B(X)\setminus \{1\}$. Then $z = x\circ (x\circ z)$ if and only if $x= z$.
\end{lem}
\begin{proof}
Suppose that $x\not = z$ and $z = x\circ (x\circ z)$. Consider $s(z) = z_1z_2...z_nz_{n-1}...z_1$ and $s(x\circ z) = u_1u_2...u_mu_{m-1}...u_1$. Then

\begin{eqnarray}
\label{lnuc1}
\pi(u_1...u_m...u_1) = x\circ z = \pi(xz_1...z_n...z_1),
\\
\label{lnuc2}
\pi(xu_1...u_m...u_1) = x\circ (x\circ z) = z = \pi(z_1...z_n...z_1)
\end{eqnarray}

By \eqref{lnuc1} and \eqref{lnuc2}, we get $\pi(u_1...u_m...u_1) = \pi(z_1...z_n...z_1u_1...u_m...u_1z_1...z_n...z_1)$. Then $m = n$ and $u_i = z_i$, for all $i$, by Corollary \ref{12}. Therefore $x = 1$, a contradiction.
\end{proof}

As a consequence of Lemma \ref{lemanuc}, we have that $(x\circ(x\circ z))\circ z \not = 1 = x\circ ((x\circ z)\circ z)$, for every $x,z\in B(X)\setminus \{1\}$ such that $x\not = z$. It follows that   $ N_ \lambda (B), N_ \mu (B) $ and $ N_ \rho (B)$ contain only the identity element $1$. Therefore we established the following result.

\begin{cor}
\label{cornuc} The nuclei and the center of $B(X)$ are trivial.
\end{cor}

\section{Open problems}

We finish this paper with two conjectures.

If $|X|>1$, it is easy to construct proper subloops of $B(X)$ that are free Bol loops of exponent $2$. In the case of free loops (infinite exponent), it is well known that all subloops of these loops are free \cite[Corollary 1, pg. 16]{B71}.

\begin{conj}
\label{conj1}
Every subloop of a free Bol loop of exponent two is free.
\end{conj}

Let $Y = \{y_1, y_2 , ..., y_n \}$ be a free set of generators of $B(X)$. For $i\in \{1,2,...,n\}$ and $v \in \left\langle Y \setminus\{y_i\}\right\rangle$, define $e_{(i,v)},f_{(i,v)}: B(X) \to B(X)$ by 

\begin{center}
$e_{(i,v)}(y_i) = y_i  v$, $f_{(i,v)}(y_i) = v  y_i$ and $e_{(i,v)}(y_j) = f_{(i,v)}(y_j) = y_j$,
\end{center}

for every $j\in \{1,2,...,n\}\setminus\{i\}$. The mappings $e_{(i,v)}$ and $f_{(i,v)}$ are automorphisms of $B(X)$ and they are called \emph{elementary automorphisms} of $B(X)$. An automorphism of $B(X)$ is called \emph{tame} if it belongs to the group generated by all elementary automorphisms of $B(X)$. A question concerning free objects in varieties of loops is whether all of their automorphisms are tame. For free Steiner loops the answer to this question is positive \cite[Theorem~7]{GRRS15}.

\begin{conj}
\label{conj2}
Every automorphism of a free Bol loop of exponent two is $tame$.
\end{conj}

\vspace{1cm}

(A. Grishkov) Instituto de Matem\'atica e Estat\'istica, Universidade de S\~ao Paulo, Rua do mat\~ao 1010, S\~ao Paulo - SP, 05508-090, Brazil
and Omsk State a.m. F.M.Dostoevsky University, Russia.

\textit{E-mail adress}: grishkov@ime.usp.br

(M. Rasskazova) Omsk State Technic University, Omsk, Russia

\textit{E-mail adress}: marinarasskazova@yandex.ru

(G. Souza dos Anjos) Instituto de Matem\'atica e Estat\'istica, Universidade de S\~ao Paulo, Rua do mat\~ao 1010, S\~ao Paulo - SP, 05508-090, Brazil

\textit{E-mail adress}: giliard.anjos@unesp.br

\end{document}